\documentclass{amsart}
\usepackage{amssymb,amsmath,amsthm,amstext,amsfonts}
\usepackage[dvips]{graphicx}
\usepackage{psfrag}
\theoremstyle{plain}
\newtheorem{maintheorem}{Theorem}

\newtheorem{maincorollary}{Corollary}

\newtheorem{mainproposition}{Proposition}

\newtheorem{theorem}{Theorem }[section]

\newtheorem{lemma}[theorem]{Lemma}

\theoremstyle{definition} \theoremstyle{remark}

\newcommand{\diam}{\operatorname{diam}}

\newcommand{\al} {\alpha}       
        
\newcommand{\ga} {\gamma}    
\newcommand{\de} {\delta}       

\newcommand{\vep}{\varepsilon}

\newcommand{\la} {\lambda}      \newcommand{\La}{\Lambda}

\newcommand{\un}{\underbar}

\newcommand{\cR}{\mathcal{R}}

\newcommand{\cQ}{\mathcal{Q}}

\newcommand{\cQn}{\cQ^{(n)}}

\newcommand{\ov}{\overline}

\begin{document}
\title{Entropy and Poincar\'e recurrence from a geometrical viewpoint}
\author{Paulo Varandas}
\address{Departamento de Matem\'atica, Universidade Federal da Bahia\\
Av. Ademar de Barros s/n, 40170-110 Salvador, Brazil.}
\email{paulo.varandas@ufba.br}
\date{\today}



 \maketitle

\begin{abstract}
We study Poincar\'e recurrence from a purely geometrical viewpoint.
In \cite{DW04} it was proven that the metric entropy is given by the
exponential growth rate of return times to dynamical balls. Here we
use combinatorial arguments to provide an alternative and more
direct proof of this result and to prove that minimal return times
to dynamical balls grow linearly with respect to its length. Some
relations using weighted versions of recurrence times are also
obtained for equilibrium states. Then we establish some interesting
relations between recurrence, dimension, entropy and Lyapunov
exponents of ergodic measures.
\end{abstract}

\section{Introduction}

Since it was introduced in Dynamical Systems more than fifty years
ago, entropy has become an important ingredient in the
characterization of the complexity of dynamical systems in both
topological and measure theoretical senses. From the measure
theoretical viewpoint the metric entropy of invariant measures
turned out to be a surprisingly universal concept in ergodic theory
since it appears in the study of different subjects as information
theory, Poincar\'e recurrence, and in the analysis of either local
or global complexities. Just as an illustration of its universal
nature, metric entropy is characterized as the exponential growth
rate of: the measure of decreasing partition elements and dynamical
balls (see e.g.~\cite{Man87} for the Shannon-McMillan-Breiman's
theorem and ~\cite{BK83}), the number of dynamical balls and
partition elements necessary to cover a relevant part of the phase
space (see e.g.~\cite{Ka80}), and the recurrence rate to elements of
a given partition (see e.g.~\cite{OW93}). We refer the reader
to~\cite{Ka07} for a very complete survey on the notion of entropy
in dynamical systems.

A particularly interesting and deep connection is the one
established between metric entropy and Poincar\'e recurrence. Given
a measurable dynamical system $f$, it follows by pioneering work of
Poincar\'e that the set of recurrent points has full probability.
This means that the iterates of almost every point (with respect to
an arbitrary invariant probability measure $\mu$) will return
arbitrarily close to itself. In particular, for any positive measure
set $A$ the function
$$
R_A(x)=\inf\{k \geq 1 :  f^k(x) \in A\}
$$
is finite almost everywhere in $A$. Given a decreasing sequence of
partitions $U_n$ it is natural to look for a limiting behavior of
the return times $R_{U_n}$ in finer scales. Such a limiting behavior
turned out to exist for ergodic stationary processes and it
coincides with the metric entropy of the system. More precisely,
Ornstein and Weiss~\cite{OW93} proved that the entropy
$h_\mu(f,\cQ)$ of an ergodic measure $\mu$ with respect to a
partition $\cQ$ is given by the (almost everywhere) well defined
limit
\begin{equation}\label{eq.Ornstein.Weiss}
h_\mu(f,\cQ)
    =\lim_{n\to \infty} \frac1n \log R_n(x,\cQ),
\end{equation}
where $R_n(x,\cQ)=\inf\{k \geq 1 : f^k(x) \in \cQn(x)\}$ is the
\emph{$n$th return time} (with respect to the partition $\cQ$),
$\cQ^{(n)}=\bigvee_{j=0}^{n-1} f^{-j}\cQ$ is the dynamically
generated partition, and $\cQ^{(n)}(x)$ denotes the element of
$\cQ^{(n)}$ that contains the point $x$. Consequently, the metric
entropy is the supremum of the exponential growth rates of
Poincar\'e recurrences over all possible choices of partitions.
Moreover, when return times are weighted with respect to some
potential we can recover estimates concerning the pressure. Some
interesting formulas concerning pressure and return times were also
obtained in~\cite{MSUZ03}.

Also very important is the notion of minimal return times that we
now describe. By Poincar\'e recurrence theorem, for every invariant
probability measure $\mu$ the \emph{minimal return time} $S(A)$ to
any positive measure set $A$ defined by
$$
S(A)=\inf\{k \geq 1 :  f^{-k}(A) \cap A \neq \emptyset\}
$$
is finite. Heuristically it is natural to expect the limiting
behavior of $S(U_n)$ in finer scales $U_n$, if it exists, to be
simpler than the the one presented by regular return times
$R_{U_n}$. In fact, Afraimovich, Chazottes, Saussol~\cite{ACS03}
proved that maps with a specification property satisfy
$$
\lim_{n\to \infty} \frac{S_n(x,\cQ)}{n}=1
    \quad \text{for $\mu$-almost every}\; x
$$
for every ergodic measure $\mu$ (provided that $h_\mu(f,\cQ)>0$),
where $S_n(x,\cQ)$ denotes the \emph{$n$th minimal return time} to
the partition element $\cQn(x)$. So, contrary to the exponential
growth presented by regular return times, minimal return times grow
linearly with $n$, i.e. the time needed for a cylinder to return to
itself is asymptotically given by its size.

To the best of our knowledge, the class of dynamical systems for
which return time statistics are studied are mostly those that
present some (finite or countable) reference partition with some
Markovian property or such that the bounded distortion property
holds. We refer the reader to ~\cite{CG93,Pac00,BV03,BT07b,Va08b}
just to quote some recent contributions. However, the existence of
such partitions constitutes itself a problem even in a context of
non-uniform hyperbolicity. We refer the reader to \cite{Pi07} for a
recent important contribution on the construction of such partitions
for nonuniformly expanding maps.

So, we turn our attention to return times to purely geometrical
objects as (regular and dynamically defined) balls. In fact, not
only regular and dynamically defined balls arise naturally in a
non-uniformly hyperbolic context as the study of Poincar\'e
recurrence to these purely geometrical objects encloses much
information about invariant measures. Given an invariant measure
$\mu$ the upper and lower pointwise dimensions $\ov d_\mu(x)$ and
$\un d_\mu(x)$ are defined by the limits
\begin{equation*}\label{def.pointwise.dimensions}
\ov d_\mu(x)
    = \limsup_{r\to 0} \frac{\log \mu(B(x,r))}{\log r}
        \quad\text{and}\quad
\underline d_\mu(x)
    = \liminf_{r\to 0} \frac{\log \mu(B(x,r))}{\log r}.
\end{equation*}
Its Hausdorff dimension $\dim_H(\mu)$, defined as the infimum of the
Hausdorff dimension of sets of full $\mu$-measure, satisfies
$\underline d_\mu(x) \le \dim_H(\mu) \le \ov d_\mu(x)$ (see e.g.
\cite{Pe97}). Barreira, Pesin, Schmeling~\cite{BPS99} proved that
any hyperbolic and ergodic measure $\mu$ of a $C^{1+\alpha}$
diffeomorphism is exact dimensional, i.e. the upper and lower
pointwise dimensions $\ov d_\mu(x)$ and $\un d_\mu(x)$ do exist and
coincide almost everywhere. By~\cite{You82} the limit is equal to
$\dim_H(\mu)$.
Moreover, Barreira, Saussol~\cite{BaSa01} proved that the pointwise
recurrence rates
\begin{equation*}\label{def.pointwise.recurrences}
\ov R(x)
    = \limsup_{\substack{r\to 0}} \frac{\log R_{B(x,r)}(x)}{-\log r}
        \quad\text{and}\quad
\underline R(x)
    = \liminf_{\substack{r\to 0}} \frac{\log R_{B(x,r)}(x)}{-\log r}
\end{equation*}
satisfy $\ov R(x) \le \ov d_\mu(x)$ and $\underline R(x) \le
\underline d_\mu(x)$ in general, and conjectured that any
$C^{1+\al}$ diffeomorphism $f$ and any hyperbolic ergodic measure
$\mu$ should satisfy
\begin{equation}\label{eq.conjBS}
\underline R(x)=\ov R(x)=\dim_H(\mu),
    \quad \text{$\mu$-almost everywhere.}
\end{equation}
We note that \eqref{eq.conjBS} was proved to hold for dynamical
systems that either present some \emph{hyperbolicity} (e.g. Axiom A
diffeomorphisms ~\cite{BaSa01} and piecewise monotone interval maps
whose derivative has $p$-variation~\cite{STV02}) or that satisfy a
\emph{rapidly mixing property} as in~\cite{Sau06}.
In~\cite{STV02,STV03} the minimal recurrence rates
\begin{equation*}\label{def.local.recurrences}
\ov S(x)
    = \limsup_{\substack{r\to 0}} \frac{S(B(x,r))}{-\log r}
        \quad\text{and}\quad
\underline S(x)
    = \liminf_{\substack{r\to 0}} \frac{S(B(x,r))}{-\log r}
\end{equation*}
are studied. In the case of endomorphisms it is shown that, if $\mu$
is a positive entropy ergodic measure and $\la_\mu,\La_\mu>0$ denote
respectively the smallest and the largest Lyapunov exponents of
$(f,\mu)$ then $\underline S(x) \ge 1/\La_\mu$ and, under some
specification property, that $\ov S(x) \le 1/\la_\mu$. In
particular, a wide family of piecewise monotone interval maps with
$p$-variation satisfy
\begin{equation}\label{eq.STV}
\underline S(x)=\ov S(x)=1/\la_\mu
    \quad \text{for $\mu$-almost every $x$},
\end{equation}
where $\la_\mu$ denotes the Lyapunov exponent of $\mu$.

Our purpose is to study return times to more natural topological
objects than partitions. Indeed, we characterize the metric entropy
as the exponential growth rate of return times to dynamical balls
and show that minimal return times to dynamical balls grow linearly
with respect to its length. These constitute geometrical
counterparts to some results in \cite{OW93} and \cite{ACS03}.
Afterwards these results are used to establish some new results
relating recurrence, dimension and Lyapunov exponents.
Although our first result appeared previously in \cite{DW04} as a
consequence of a generalization of Shannon-McMillan-Breiman's
theorem, computing the exponential decreasing rate of the measure of
partition elements determined when a point enters a given set, we
believe that the combinatorial arguments used here can be applied in
some different contexts as in the study of hitting time statistics
and fluctuations of return times. In fact, one expect the
fluctuations of the return times in Theorem~\ref{Thm} to be
log-normal with respect to any measure satisfying a weak Gibbs
property as in~\cite{Yu99,Va08b,VV1}. Using~\cite{Sau01} this is the
case provided exponential hitting time statistics. However, to the
best of our knowledge, there are no known examples where exponential
return time statistics to dynamical balls has been obtained for
(multidimensional) dynamical systems outside of the uniformly
hyperbolic setting.
It seems that the combinatorial arguments used here can be useful to
recover exponential return time statistics for dynamical balls from
the corresponding result for elements of some relevant partition. In
particular, this should apply to the non-uniformly hyperbolic maps
considered in~\cite{OV08,Va08b}.

This paper is organized as follows. In Section~\ref{s.statement} we
present the main result. Some definitions and preliminaries are
presented in Section~\ref{s.preliminaries}. In
Section~\ref{s.recurrence} we study regular and return times to
dynamical balls. The proofs of Theorems~\ref{Thm} and~\ref{thm} are
given in Subsections~\ref{s.Proof} and~\ref{s.proof} respectively.
Finally, in Section~\ref{s.applications} we apply the previous
results to study dimension of ergodic measures and prove
Proposition~\ref{p.application1} and Theorem~\ref{thm.application}.

\section{Statement of the main results}\label{s.statement}

In this section we introduce some necessary definitions and state
our main results. Throughout, assume that $X$ is a compact metric
space and let $f: X \to X$ be a continuous transformation. Given
$\vep>0$ and $n\ge 1$ the \emph{dynamical ball} $B(x,n,\vep)$ is the
set
 $
 B(x,n,\vep)
    =\{y \in X : d(f^j(x),f^j(y)) < \vep \; \text{for every} \; 0 \leq j \leq n-1 \}.
 $
We define the \emph{$n$th return time $R_n(x,\vep)$} to the
dynamical ball $B(x,n,\vep)$ by
$$
R_n(x,\vep)
    =\inf\{k \geq 1 : f^k(x) \in B(x,n,\vep)\}.
$$
We recall the following result that follows from more general result
in \cite{DW04}:

\begin{maintheorem}\label{Thm}
Let $\mu$ be an ergodic $f$-invariant probability measure. The
limits
$$
\overline h(f,x)
    =\lim_{\vep\to 0} \limsup_{n \to \infty} \frac1n \log R_n(x,\vep)
\quad\text{and}\quad \underline h(f,x)
    =\lim_{\vep\to 0} \liminf_{n \to \infty} \frac1n \log R_n(x,\vep)
$$
exist for $\mu$-almost every $x$ and coincide with the metric
entropy $h_\mu(f)$.
\end{maintheorem}

Let us comment on the assumption of ergodicity in the theorem above.
By ergodic decomposition every $f$-invariant probability measure
$\mu$ can be decomposed as a convex combination of ergodic measures
$\mu_x$, i.e.
 $
\mu= \int \mu_x \; d\mu(x).
 $
Moreover, since the metric entropy map is affine then $h_\mu(f)=
\int h_{\mu_x}(f) \; d\mu(x)$. So, applying Theorem~\ref{Thm} to
each ergodic component $\mu_x$ and integrating with respect to $\mu$
we obtain the following immediate consequence.

\begin{maincorollary}
If $\mu$ is an $f$-invariant probability measure then the limits
$\overline h(f,x)$ and $\underline h(f,x)$ defined above do exist
for $\mu$-almost every $x$. Moreover, the metric entropy $h_\mu(f)$
satisfies
$$
h_\mu(f)= \int \overline h(f,x) \; d\mu(x) = \int \underline h(f,x)
\; d\mu(x).
$$
\end{maincorollary}

Given a continuous potential $\phi : X \to \mathbb R$ the
\emph{metric pressure} $P_\mu(f,\phi)=h_\mu(f)+\int \phi \;d\mu$ of
the invariant measure $\mu$ with respect to $f$ and $\phi$ can also
be written using weighted recurrence times. This is a consequence of
Birkhoff's ergodic theorem and Theorem~\ref{Thm} as we now explain.
Indeed, given an $f$-invariant and ergodic probability measure $\mu$
there exists a full measure set $\cR$ such that
$$
\lim_{n\to \infty} \frac1n S_n\phi(x)=\int \phi\,d\mu
    \quad\text{and}\quad
h_\mu(f)=
    \lim_{\vep\to 0} \limsup_{n \to \infty} \frac1n \log R_n(x,\vep)
$$
for every $x\in \cR$. Given any $\de>0$ and $x\in \cR$ it follows
from uniform continuity of $\phi$ the existence of $\vep_\de>0$ such
that $|S_n\phi(B(x,n,\vep))-S_n\phi(x)|< \de n$ for every $n \ge 1$
and every $0<\vep<\vep_\de$, where $S_n \phi(B(x,n,\vep))= \sup
\{\sum_{j=0}^{n-1} \phi(f^j(y)): y \in B(x,n,\vep)\}$. In
consequence,
$$
\left|
    \limsup_{n\to \infty} \left[\frac1n S_n\phi(B(x,n,\vep)) +\frac1n \log R_n(x,\vep)\right]
    -\left(h_\mu(f)+\int \phi\,d\mu \right)
\right|
    <2\de
$$
for every small $\vep>0$. Hence we deduced the following result,
relating the metric pressure with appropriate weighted return times.

\begin{maincorollary}
Let $\mu$ be an $f$-invariant and ergodic probability measure. Then
$$
P_\mu(f,\phi)
    =\lim_{\vep\to 0} \limsup_{n\to\infty}
        \frac1n \log \Big[ e^{S_n \phi(B(x,n,\vep))}R_n(x,\vep)\Big],
        \;\;\text{for $\mu$-a.e. $x$}.
$$
\end{maincorollary}

Now we turn our attention to minimal return times. We define the
\emph{$n$th minimal return time} $S_n(x,\vep)$ to the dynamical ball
$B(x,n,\vep)$ by
$$
S_n(x,\vep)
    =\inf \{k \ge 1 : f^{-k}(B(x,n,\vep)) \cap B(x,n,\vep)\neq \emptyset\}
$$
Clearly $S_n(x,\vep) \le R_n(x,\vep)$ and so these minimal return
times are finite in a set of total probability. Moreover, we will
prove that minimal return times $S_n$ to dynamical balls grow
linearly with $n$. First we recall a definition. We say that $f$
satisfies the \emph{specification property} if, given $\de>0$ there
is an integer $N \geq 1$ such that the following holds: for any
$k\ge 1$, any points $x_1, \dots, x_k$, any integers $0 = a_1\le
b_1<a_2 \le b_2 < \dots < a_k \le b_k$ satisfying
$a_{i+1}-b_{i}>N(\de)$ and any integer $p \ge b_k + N(\de)$ there
exists a point $x\in X$ such that $f^p(x)=x$ and $d(f^j(x),f^j(x_i))
<\de$ for every $a_i\leq j \leq b_i$ and $1 \le i \le k$. Our second
main result is as follows.

\begin{maintheorem}\label{thm}
Assume that $f$ has the specification property. If $\mu$ is an
$f$-invariant, ergodic measure such that $h_\mu(f)>0$, the limits
$$
\ov S(x)=\lim_{\vep\to 0}\limsup_{n\to \infty} \frac1n S_n(x,\vep)
    \quad\text{and}\quad
\un S(x)=\lim_{\vep\to 0}\liminf_{n\to \infty} \frac1n S_n(x,\vep)
$$
exist and are equal to one for $\mu$-almost every $x$.
\end{maintheorem}

It is not hard to check that this result also holds true if $\mu$
satisfies the nonuniform specification property of \cite{STV03}.
However we shall not use or prove this fact.
The final part of this section is devoted to the discussion of the
relation between entropy, dimension and Lyapunov exponents.

\begin{mainproposition}\label{p.application1}
Assume that $f: X \to X$ is a continuous transformation and that
there exist constants $\de,\la,\La>0$ such that
 $
\la \, d(x,y) \leq d(f(x),f(y)) \leq \La \, d(x,y)
 $
for every $x,y \in X$ so that $d(x,y)<\de$. If $\mu$ is an
$f$-invariant ergodic probability measure with positive entropy then
$$
\frac{h_\mu(f)}{\log\La} \le \underline R(x)
    \quad\text{and}\quad
\ov R(x) \le \frac{h_\mu(f)}{\log\la},
$$
and $1/\log\La \le \underline S(x)$ for $\mu$-almost every $x$. If,
in addition, $f$ satisfies the specification property then $\ov S(x)
\le 1/\log\la$ for $\mu$-a.e. $x$.
\end{mainproposition}

If $f$ is a linear, conformal expanding tori endomorphism it
satisfies the specification property and there exists $\la>1$ so
that $d(f(x),f(y))=\la d(x,y)$ for every close $x,y\in X$. Moreover,
if $\mu$ is an ergodic measure its Lyapunov exponent is
$\la_\mu=\log\la$. Using that $\dim_H(\mu)=h_\mu(f)/\la_\mu$ (see
e.g.~\cite{You82}) we obtain:

\begin{maincorollary}
Let $f: \mathbb T^n \to \mathbb T^n$ be a linear, conformal
expanding tori endomorphism. If $\mu$ is a positive entropy ergodic
$f$-invariant probability measure then \eqref{eq.conjBS} and
\eqref{eq.STV} hold.
\end{maincorollary}

The following result is an asymptotic version of
Proposition~\ref{p.application1} above for differentiable
endomorphisms.

\begin{maintheorem}\label{thm.application}
Assume that $f: M \to M$ is a $C^{1+\al}$ endomorphism, $\mu$ is a
positive entropy $f$-invariant, ergodic probability measure and that
$0<\la_1\leq \la_2 \leq \dots \leq \la_d$ are the Lyapunov exponents
of $(f,\mu)$. Then $h_\mu(f)/\la_d \le \underline R(x)$,
$h_\mu(f)/\la_1 \ge \overline R(x)$ and $1/\la_d \le \underline
S(x)$ for $\mu$-almost every $x$. If, in addition, $f$ satisfies the
specification property then $\ov S(x) \le 1/\la_1$ for $\mu$-almost
every $x$.
\end{maintheorem}

Since topologically mixing continuous interval maps satisfy the
specification property then we get:

\begin{maincorollary}
Let $f: I \to I$ be a topologically mixing $C^{1+\al}$ interval map
and assume that $\mu$ is a positive entropy, ergodic, $f$-invariant
probability measure. Then \eqref{eq.conjBS} and \eqref{eq.STV} hold.
\end{maincorollary}

\section{Preliminaries}\label{s.preliminaries}

\subsection{Metric entropy}

We recall some characterizations of metric entropy. The first one is
due to Katok~\cite[Theorem I.I]{Ka80}. Given $0<c<1$, denote by
$N(n,\vep,c)$ the minimum number of dynamical balls necessary to
cover a set of measure $c$. Indeed, if $\mu$ is ergodic Katok proved
that for every $c\in(0,1)$
\begin{equation}\label{eq.Katok}
h_\mu(f)
    =\lim_{\vep\to 0} \liminf_{n \to \infty} \frac1n \log N(n,\vep,c)
    =\lim_{\vep\to 0} \limsup_{n \to \infty} \frac1n \log N(n,\vep,c).
\end{equation}
Using Shannon-McMillan-Breiman's theorem and arguments analogous to
the ones used in the proof of \eqref{eq.Katok} it is straightforward
to check the following property.

\begin{lemma} Let $\cQ$ be a partition on $X$ and $c\in(0,1)$ be given. Then
\begin{equation}\label{eq.partition.number}
h_\mu(f,\cQ)
    =\lim_{\vep\to 0} \liminf_{n \to \infty} \frac1n \log N(n,\cQ,c)
    =\lim_{\vep\to 0} \limsup_{n \to \infty} \frac1n \log N(n,\cQ,c),
\end{equation}
where $N(n,\cQ,c)$ denotes the minimum number of $n$-cylinders of
the partition $\cQn$ necessary to cover a set of measure $c$.
\end{lemma}

\subsection{Combinatorial lemma}

In this subsection we prove the following covering lemma for
dynamical balls associated with points with slow recurrence to the
boundary of a given partition.

\begin{lemma}~\label{l.prelim1}
Let $\cQ$ be a finite partition of $X$ and consider $\vep>0$
arbitrary. Let $V_{\vep}$ denote the $\vep$-neighborhood of the
boundary $\partial\cQ$. For any $\al>0$ there exists $\gamma>0$
(depending only on $\alpha$) such that for every $x \in X$
satisfying $\sum_{j=0}^{n-1} 1_{V_{\vep}}(f^j(x)) < \gamma n$ the
dynamical ball $B(x,n,\vep)$ can be covered by $e^{\alpha n}$
cylinders of $\cQ^{(n)}$.
\end{lemma}

\begin{proof}
Fix an arbitrary $\alpha>0$. Since $B(z,\vep) \subset \cQ(z)$ for
every $z\not\in V_\vep$, the itinerary of any point $y$ in the
dynamical ball $B(x,n,\vep)$ centered at a point $x \in X$
satisfying $\sum_{j=0}^{n-1} 1_{V_{\vep}}(f^j(x)) < \gamma n$ will
differ from the one of $x$ by at most $[\gamma n]$ choices of
partition elements. Since there are at most
 $
 \left(
\begin{array}{c}
n \\
 \gamma n
\end{array}
 \right)
 \; \; (\# \cQ)^{\gamma n}
 $
such choices, this can be made smaller than $e^{\alpha n}$ provided
that $\gamma>0$ is small enough. This completes the proof of the
lemma.
\end{proof}

\section{Dynamical balls and recurrence}\label{s.recurrence}

In this section our purpose is to prove Theorems~\ref{Thm} and
~\ref{thm} that relate entropy with the usual and minimal return
times to dynamical balls.

\subsection{Proof of Theorem~\ref{Thm}}\label{s.Proof}

We begin the proof of the theorem by noting that the limits in the
statement of Theorem~\ref{Thm} are indeed are well defined almost
everywhere. Given $n \ge 1$, $\vep>0$ and $x\in X$ it holds that
$R_n(x,\vep) \geq R_{n-1}(f(x),\vep)$. Indeed, $f^{R_n(x,\vep)}(x)
\in B(x,n,\vep)$ implies that $f^{R_n(x,\vep)}(f(x)) \in
f(B(x,n,\vep)) \subset B(f(x),n-1,\vep)$, which proves our claim.
Define
\begin{equation*}
\underline h(x,\vep)=\liminf_{n\to \infty} \frac1n \log R_n(x,\vep)
    \quad\text{and}\quad
\overline h(x,\vep)=\limsup_{n\to \infty} \frac1n \log R_n(x,\vep).
\end{equation*}
It follows from the discussion in the previous paragraph that
$\underline h(f(x),\vep) \leq \underline h(x,\vep)$ and $\overline
h(f(x),\vep) \leq \overline h(x,\vep)$. Since $\mu$ is ergodic these
functions are almost everywhere constant and their values will be
denoted by $\underline h(\vep)$ and $\overline h(\vep)$,
respectively. Denote by $\un h(f)$ and by $\ov h(f)$ the limits when
$\vep\to 0$ of the functions $\un h(\vep)$ and $\ov h(\vep)$. Such
limits do exist by monotonicity of the previous functions on $\vep$.
Hence, to prove the theorem it is enough to show that
\begin{equation}\label{eq.main.estimates}
\overline h(f) \le h_\mu(f) \le \underline h(f).
\end{equation}

To deal with the left hand side inequality in
\eqref{eq.main.estimates}, let $\vep>0$ be fixed and pick any
partition $\cQ$ satisfying $\mu(\partial\cQ)=0$ and
$\diam(\cQ)<\vep$. By construction we get that $B(x,n,\vep) \supset
\cQ_n(x)$ for $\mu$-almost every $x$ and every $n \ge 1$.
Consequently, $R_n(\cdot ,\cQ) \geq R_n(\cdot,\vep)$ and, using
Ornstein-Weiss's theorem,
$$
h_\mu(f) \geq h_{\mu}(f,\cQ)
        \geq \limsup_{n\to\infty} \frac1n \log R_n(x,\vep)
$$
for $\mu$-a.e. $x$. Since $\vep$ was chosen arbitrary one gets that
$h_\mu(f) \ge \ov h(f)$ as claimed.

We are left to prove the second inequality in
\eqref{eq.main.estimates}. Assume, by contradiction, that $h_\mu(f)
> \un h(f)$ and pick a finite partition $\cQ$ such that
$\mu(\partial \cQ)=0$ and $h_\mu(f)\geq h_\mu(f,\cQ) > b> a
> \underline h(f)$. Fix $0<\gamma<(b-a)/6$ small such that
Lemma~\ref{l.prelim1} holds for $\alpha=(b-a)/2$.
For every sufficiently small $\vep>0$, if $V_{\vep}$ denotes the
$\vep$-neighborhood of the boundary $\partial\cQ$ then
$\mu(V_{\vep}) < \gamma/2$. By ergodicity and Birkhoff's ergodic
theorem we may choose $N_0 \ge 1$ large such that the set
\begin{equation}\label{eq.An0}
 A= \Big\{x \in X : \sum_{j=0}^{n-1} 1_{V_{\vep}}(f^j(x)) < \gamma n, \forall n \ge N_0\Big\}
\end{equation}
has measure larger than $1-\gamma$. By Lemma~\ref{l.prelim1} each
dynamical ball $B(z,\ell,\vep)$ of length $\ell \ge N_0$ centered at
any point $z\in A$ can be covered by $e^{\alpha \,\ell}$ cylinders
of $\cQ^{(\ell)}$.
Furthermore, provided that $N_1\ge N_0$ is large enough, the measure
of the set
$$
 B=\Big\{x \in X : \exists N_0 \le n \le N_1 \;\text{s.t.}\; R_n(x,\vep) \le e^{a n} \Big\}
$$
is also larger than $1-\gamma$. For notational simplicity we shall
omit the dependence of the sets $A$ and $B$ on the integers $N_0$
and $N_1$. Using once more Birkhoff's ergodic theorem, we may take
$N_2 \ge 1$ large enough so that
$$
\Lambda
    = \Big\{x \in X : \sum_{j=0}^{k-1} 1_{A\cap B}(f^j(x)) > (1-3\gamma) k, \forall k \ge N_2\Big\}
$$
has measure at least $1/2$. We claim that there exists a constant
$C>0$ such that $\Lambda$ is covered by $C e^{bk}$ cylinders of
$\cQ^{(k)}$, for every large $k$. This will imply that
$$
h_\mu(f,\cQ)= \lim_{n\to\infty} \frac1n \log N(k,\cQ,1/2) < b,
$$
leading to a contradiction that will complete the proof of the
theorem.

Fix $x\in \Lambda$ and $k \gg N_2$. We proceed to divide the set
$\{0,1,2,\dots, k\}$ into blocks according to the recurrence
properties of the orbit of $x$. If $x \not\in A\cap B$ then we
consider the block $[0]$. Otherwise, we take the first integer $N_0
\leq m \leq N_1$ such that $R_m(x,\vep) \leq e^{a m}$ and consider
the block $[0,1,\dots, m-1]$. We proceed recursively and, if $\{1,
\dots, k'\}$ $(k'<k)$ is partitioned into blocks then the next block
is $[k'+1]$ if $f^{k'+1}(x) \not\in A\cap B$ and it will be $[k'+1,
k'+2, \dots, k'+m']$ if $f^{k'+1}(x)\in A\cap B$ and $m'$ is the
first integer in $[N_0,N_1]$ such that $R_{m'}(f^{k'+1}(x),\vep)
\leq e^{am'}$. This process will finish after a finite number of
steps and partitions $\{1, 2, \dots, k\}$ according to the
recurrence properties of the iterates of $x$, except possibly the
last block which has size at most $N_1$. We write the list of
sequence of block lengths determined above as
$\iota(x)=[m_1,m_2,\dots, m_{i(x)}]$. By construction there are at
most $3\gamma k$ blocks of size one.
This enable us to give an upper bound on the number of $k$-cylinders
$\cQ^{(k)}$ necessary to cover $\Lambda$. First note that since each
$m_i$ is either one or larger than $N_0$ then there are at most
$k/N_0$ blocks of size larger than $N_0$. Hence there are at most
$$
\sum_{j \le 3\gamma k}
 \left(\begin{array}{c}
    \frac{k}{N_0} +3\gamma k \\ j
 \end{array}\right)
    \le 3\gamma k
    \left(\begin{array}{c}
    \frac{k}{N_0} +3\gamma k \\ 3\gamma k
 \end{array}\right)
$$
possibilities to arrange the blocks of size one. Now, we give an
estimate on the number of possible combinatorics for every prefixed
configuration $\iota=[m_1, m_2, \dots, m_\ell]$, satisfying $\sum
m_j=k$ and $\#\{ j : m_j =1\} <3\gamma k$. This will be done fixing
elements from the right to the left. Define $M_j=\sum_{i\le j} m_j$.
If $x \in \Lambda$ is such that $\iota(x)=\iota$ there are at most
$\#\cQ$ possibilities to choose a symbol for each block of size one.
Moreover, if $1 \le \kappa \le \ell$ is the first integer such that
$\sum_{j=\kappa+1}^\ell m_i < N_1+e^{a N_1}$ then there are at most
$(\#\cQ)^{(\ell-\kappa)N_1} \leq (\#\cQ)^{N_1(1+N_1+e^{a N_1})}$
possibilities for choices of $(m_\kappa+m_{\kappa+1}+ \dots
+m_\ell)$-cylinders with combinatorics $[m_\kappa, \dots, m_\ell]$.
Recall that $R_{m_{k-1}}(f^{M_{\kappa-2}}(x), \vep) \leq e^{a
m_{k-1}} \leq e^{a N_1}$ and, by Lemma~\ref{l.prelim1}, the
dynamical ball $B(f^{M_{\kappa-2}}(x),m_{\kappa-1},\vep)$ is
contained in at most $e^{\alpha m_{\kappa-1}}$ cylinders in
$\cQ^{(m_{\kappa-1})}$. Hence the possible itineraries for the
$m_{\kappa-1}$ iterates $\{f^{M_{\kappa-2}}(x), \dots,
f^{M_{\kappa-1}}(x)\}$ may be chosen among $e^{\alpha m_{\kappa-1}}$
options corresponding to each of the $e^{a m_{\kappa-1}}$ previously
possibly distinct and fixed  blocks of size $m_{\kappa-1}$ in
$[m_\kappa, \dots, m_\ell]$. This shows that there are at most
$e^{(a+\alpha)m_{\kappa-1}}$ possible itineraries for the
$m_{\kappa-1}$ iterations of $f^{M_{\kappa-2}}(x)$. Proceeding
recursively for $m_{\kappa-2}, \dots, m_2,m_1$ we conclude, after
some finite number of steps, that there exists $C>0$ (depending only
on $N_1$) such that if $\gamma$ was chosen small then $\Lambda$ can
be covered by
$$
 3\ga k \left(\begin{array}{c}
 \frac{k}{N_0} +3\gamma k \\ 3\gamma k
 \end{array}\right)
    (\#\cQ)^{N_1(1+N_1+e^{aN_1})}
        (\#\cQ)^{3\gamma k}
            e^{(a+\alpha)k}
 \leq C e^{bk}
$$
cylinders in $\cQ^{(k)}$. This proves the claim and finishes the
proof of the theorem.

\subsection{Proof of Theorem~\ref{thm}}\label{s.proof}

The proof of the theorem is divided in two steps. On the one hand,
the specification property guarantees that for every small $\vep>0$
there exists an integer $N(\vep)\ge 1$ such that for any $x\in X$
and $n \ge N(\vep)$ there is some periodic point of period smaller
or equal to $n+N(\vep)$ in $B(x,n,\vep)$. Consequently,
 $
\limsup_{n\to\infty} \frac1n S_n(x,\vep)
    \le 1
 $
for every small $\vep>0$, and proves that $\ov S(x) \le 1$ almost
everywhere.

So, to prove the theorem it remains to show that $\un S(x) \ge 1$
for $\mu$-almost every $x$. We claim that for any $\eta<1$ there
exists a measurable set $E_\eta$ such that $\mu(E_\eta)>1-\eta$ and
$\mu( x \in E_\eta : S_n(x,\vep) \leq \eta \,n )$ is summable for
every small $\vep$. Using Borel-Cantelli lemma it will follow that
any point $x \in E_\eta$ satisfies $S_n(x,\vep)
> \eta \,n$ for all but finitely many values of $n$ and every small
$\vep$. The result will follow from the arbitrariness of $\eta$.
The remaining of this paragraph is devoted to the proof of the
previous claim. Let $\eta \in (0,1)$ be arbitrary and fix a small
$0<\al<\frac13 (1-\eta)h_{\mu}(f)$. Consider a finite partition
$\cQ$ satisfying $\mu(\partial \cQ)=0$ and $3\al<(1-\eta)h$, where
$h=h_\mu(f,\cQ)>0$. If $\vep_0>0$ is small enough then
$\mu(V_\vep)<\ga/2$ for every $0<\vep<\vep_0$, for $\ga=\ga(\al)>0$
given by Lemma~\ref{l.prelim1}. Using Birkhoff's ergodic theorem,
Shannon-McMillan-Breiman's theorem and Lemma~\ref{l.prelim1}, for
almost every $x$ there exists an integer $N(x) \ge 1$ such that for
every $n \ge N(x)$
\begin{equation}\label{eq.measure}
\sum_{j=0}^{n-1} 1_{V_\vep}(f^j(x)) <\ga n
    \quad\text{and}\quad
e^{-(h+\al) n} \le
    \mu(\cQ^{(n)}(x))
\le e^{-(h-\al) n}
\end{equation}
and, consequently, any dynamical ball $B(x,n,\vep)$ is covered by a
collection $\cQ^{(n)}(x,\vep)$ of $e^{\al n}$ cylinders of the
partition $\cQ^{(n)}$. Pick $N \ge 1$ large such that set $E_\eta$
of points $x\in X$ satisfying \eqref{eq.measure} for every $n \ge N$
has measure greater than $1-\eta$. Since $\cQ$ is finite there is
$K>0$ such that
\begin{equation*}
K^{-1}e^{-(h+\al) n} \le
    \mu(\cQ^{(n)}(x))
\le K e^{-(h-\al) n}
\end{equation*}
for every $x\in E_\eta$ and every $n \ge 1$.
For $n \ge N$ we denote by $E_\eta(n,k)$ the set of points in $
E_\eta$ such that $S_n(\cdot,\vep)=k$. If $x\in E_\eta(n,k)$ then
the dynamical ball $B(x,n,\vep)$ is contained in the subcollection
of cylinders $Q_n \in \cQn(x,\vep)$ whose iteration by $f^k$
intersects any of the $n$-cylinders of $\cQn(x,\vep)$. Any such
cylinder $Q_n$ is determined by its first $k$ symbols and by the at
most $e^{\al n}$ possible strings following them. So, the number of
those cylinders is bounded by $e^{\al n}$ times the number of
cylinders in $\cQ^{(k)}$ that intersect $E_\eta$, that is, $e^{\al
n} K e^{(h+\al)k}$. Hence, if $n \ge N$
\begin{align*}
\mu\Big( x \in E_\eta : S_n(x,\vep) < \eta n\Big)
    & \le \sum_{k=0}^{\eta n} \sum_{\substack{Q_n \in \cQn \\ Q_n \cap E_\eta(n,k) \neq \emptyset}}
        \mu(Q_n) \le K \eta n \; e^{-(h-2\al) n} e^{(h+\al) \eta n },
\end{align*}
which is summable because $(h-2\al)-\eta(h+\al)>(1-\eta)h-3\al>0$.
This proves our claim and completes the proof of Theorem~\ref{thm}.

\section{Local recurrences and applications to dimension theory}\label{s.applications}

This section is devoted to the proof of
Proposition~\ref{p.application1} and Theorem~\ref{thm.application}.

\subsection{Proof of Proposition~\ref{p.application1}}

Our assumptions guarantee that $B(x,\vep\La^{-n}) \subset
B(x,n,\vep) \subset B(x,\vep \la^{-n})$ for every $x \in X$, $n \ge
1$ and every small $\vep>0$. Hence
$$
\underline R(x)
    \ge \liminf_{\vep\to 0}
        \Big[\lim_{n\to\infty} \frac{\log R_{B(x,\La^{-n}\vep)}(x)}{-\log (\vep\La^{-n})}\Big]
    \ge \frac{h_\mu(f)}{\log\La}
$$
for $\mu$-almost every $x$, using Theorem~\ref{Thm}. The proof of
the inequality $\ov R(x) \le h_\mu(f)/\log\la$ is analogous.
Moreover, using the specification property and Theorem~\ref{thm} it
also follows similarly that $1/\log\La \le \underline S(x)$ and $\ov
S(x) \le 1/\log\la$ in a set of total probability. This finishes the
proof of the proposition.

\subsection{Proof of Theorem~\ref{thm.application}}

We make use of Pesin's local charts (see e.g.
\cite[Appendix]{KH95}). Given $\eta>0$, for $\mu$-almost every $x$
there exists $q_\eta(x)\ge 1$ and an embedding $\Phi_x$ of the
neighborhood $R_x\subset \mathbb R^d$ of size $1/q_\eta(x)$ around
$0$ onto a neighborhood $U_x\subset M$ of $x$ such that:
\begin{enumerate}
\item $e^{-\eta} q_\eta(x) \le q_\eta(f(x)) \le e^{\eta} q_\eta(x)$;
\item $C^{-1}d(\Phi_x(z),\Phi_x(z')) \le |z-z'| \le q_\eta(x)\, d(\Phi_x(z),\Phi_x(z'))$
    for every $z,z' \in R_x$, for some universal constant $C$;
\item the map $f_x=\Phi_{f(x)}^{-1} \circ f \circ \Phi_{x}$
    satisfies
    \begin{itemize}
    \item[(a)] $e^{\la_1-\eta}|v| \leq |Df_x(0)v| \leq e^{\la_d+\eta}|v|, \; \forall v\in T_x M$, and
    \item[(b)] $\text{Lip}(f_x-Df_x(0))<\eta$.
    \end{itemize}
\end{enumerate}
We claim that the dynamical ball $B(x,n,\vep)$ contains the ball of
radius $r_n(x,\vep)=\vep e^{-(\la_d+3\eta)n}/ (C q_\eta(x)^{2})$
centered at $x$ for every small $\vep$.
Given $x\in M$ set $\hat x=\Phi_x^{-1}(x)$ and $\hat
f^k_x=f_{f^k(x)}\circ \dots \circ f_{f(x)} \circ f_{x}$. Indeed, if
$d(x,y)<r_n(x,\vep)$ then $\hat y=\Phi_x^{-1}(y) \in R_x$ and $|\hat
f_x(\hat x)-\hat f_x(\hat y)| \le e^{\la_d+2\eta} |\hat x-\hat y|
\ll 1/q_\eta(f(x))$. Recursively, we get that $\hat f^k_x(\hat y)
\in R_{f^k(x)}$ and
$$
d(f^k(x),f^k(y)) \le C |\hat{f}_x^k(\hat x)-\hat{f}_x^k(\hat y)|
                \le \frac{\vep}{q_\eta(f^k(x))} e^{(\la_d+3\eta)(k-n)}
                <\vep
$$
for every $0\le k \le n$, which proves our claim. Since
 $
 R_n(x,\vep)
            \le R_{B(x, r_n(x,\vep))}(x)
 $
 and
 $
 S_n(x,\vep)
            \le S(B(x, r_n(x,\vep)))
 $
for every $n$ and every $\vep$, using Theorems~\ref{Thm} and
~\ref{thm} we conclude that
$$
\underline R(x)
    =\liminf_{\vep\to 0} \frac{\log R_{B(x, r_n(x,\vep))}(x)}{-\log r_n(x,\vep)}
    \ge \liminf_{\vep\to 0}
        \Big[\lim_{n\to\infty} \frac{\log R_n(x,\vep)}{-\log r_n(x,\vep)}\Big]
    \ge \frac{h_\mu(f)}{\la_d+3\eta}.
$$
and, analogously, $\underline S(x)\ge \frac{1}{\la_d+3\eta}.$ Hence
$\underline R(x) \ge h_\mu(f)/\la_d$ and $\underline S(x) \ge
1/\la_d$, because $\eta$ was chosen arbitrary.

On the other hand, if one assumes that for almost every $x$ and
every $n$ there exists a radius $\vep_0(x,n)>0$ so that the
dynamical ball $B(x,n,\vep)$ is contained in the ball of radius
$r_n(x,\vep)=C\vep q_\eta(x)^{-1} e^{-(\la_1-3\eta)n}$ around $x$
for every $0<\vep<\vep_0$ then
$$
\ov R(x)
    =\limsup_{\vep\to 0} \frac{\log R_{B(x, r_n(x,\vep))}(x)}{-\log r_n(x,\vep)}
    \le \limsup_{\vep\to 0}
        \Big[\lim_{n\to\infty} \frac{R_n(x,\vep)}{-\log r_n(x,\vep)}\Big]
    \le \frac{h_\mu(f)}{\la_1-3\eta}
$$
and, similarly, $\ov S(x) \le 1/(\la_1-3\eta)$. So, the result will
follow by arbitrariness of $\eta$. To prove the previous claim note
that if $\vep_0(x,n)=e^{-(\la_d+3\eta)n}/ (C q_\eta(x)^{2})$ then
any $y\in B(x,n,\vep)$ satisfies $\hat f^k_x(\hat y) \in R_{f^k(x)}$
for every $0\le k \le n$ and $0<\vep<\vep_0$. Moreover,
$$
d(x,y) 
                \le C e^{-(\la_1-2\eta)n} |\hat{f}_x^n(\hat x)-\hat{f}_x^n(\hat y)|
                <C\vep q_\eta(x)^{-1} e^{-(\la_1-3\eta)n}.
$$
This proves the claim and finishes the proof of the theorem.

\vspace{.5cm}
\medskip \textbf{Acknowledgements:}
Part of this work was done during the International Conference in
Honor of Michael Misiurewicz - Bedlewo, and the School and Workshop
on Dynamical Systems - Trieste. The author is thankful to IMPAN and
ICTP for providing excellent research conditions, to V. Ara\'ujo and
A. Castro for encouragement and to J.-R. Chazottes for useful
remarks to a previous version of the paper. This work was partially
supported by FAPERJ and CNPq. 


\bibliographystyle{alpha}

\begin{thebibliography}{2}

\bibitem{ACS03}
V.~S. Afraimovich, J.R. Chazottes, and B.~Saussol.
\newblock Pointwise dimensions for {P}oincar\'e recurrence associated with maps
  and special flows.
\newblock {\em Discrete Contin. Dyn. Syst.}, 9:263--280, 2003.

\bibitem{BK83}
M.~Brin and A.~Katok.
\newblock On local entropy.
\newblock In {\em Geometric dynamics (Rio de Janeiro, 1981)}, volume 1007 of
  {\em Lecture Notes in Math.}, pages 30--38. Springer, 1983.

\bibitem{BPS99}
L.~Barreira, Ya. Pesin, and J.~Schmeling.
\newblock Dimension and product structure of hyperbolic measures.
\newblock {\em Ann. of Math.}, 149:755--783, 1999.

\bibitem{BaSa01}
L.~Barreira and B.~Saussol.
\newblock Hausdorff dimension of measures via {P}oincar\'e recurrence.
\newblock {\em Comm. Math. Phys.}, 219:443--463, 2001.

\bibitem{BT07b}
H.~Bruin and M.~Todd.
\newblock Return time statistics for invariant measures for interval maps with
  positive lyapunov exponent.
\newblock Preprint, 2007.

\bibitem{BV03}
H.~Bruin and S.~Vaienti.
\newblock Return time statistics for unimodal maps.
\newblock {\em Fund. Math.}, 176:77--94, 2003.

\bibitem{CG93}
P.~Collet and A.~Galves.
\newblock Statistics of close visits to the indifferent fixed point of an
  interval maps.
\newblock {\em J. Stat. Phys.}, 72:459--78, 1993.

\bibitem{DW04}
T.~Downarowicz and B.~Weiss.
\newblock Entropy theorems along times when {$x$} visits a set.
\newblock {\em Illinois J. Math.}, 48:59--69, 2004.

\bibitem{Ka80}
A.~Katok.
\newblock Lyapunov exponents, entropy and periodic points of diffeomorphisms.
\newblock {\em Publ. Math. IHES}, 51:137--173, 1980.

\bibitem{Ka07}
A.~Katok.
\newblock Fifty years of {E}ntropy in {D}ynamics: 1958 - 2007.
\newblock {\em Journal of Modern Dynamics}, 1:545--596, 2007.

\bibitem{KH95}
A.~Katok and B.~Hasselblatt.
\newblock {\em Introduction to the modern theory of dynamical systems}.
\newblock Cambridge University Press, 1995.

\bibitem{Man87}
R.~Ma{\~{n}}{\'{e}}.
\newblock {\em Ergodic theory and differentiable dynamics}.
\newblock 1987.

\bibitem{MSUZ03}
V.~Maume-Deschamps, B.~Schmitt, M.~Urba\'nski, and A.~Zdunik.
\newblock Pressure and recurrence.
\newblock {\em Fund. Math.}, 178:129--141, 2003.

\bibitem{OV08}
K.~Oliveira and M.~Viana.
\newblock Thermodynamical formalism for robust classes of potentials and
  non-uniformly hyperbolic maps.
\newblock {\em Ergod. Th. {\&} Dynam. Sys.}, 28, 2008.

\bibitem{OW93}
D.~Ornstein and B.~Weiss.
\newblock Entropy and data compression schemes.
\newblock {\em IEEE Trans. Inform. Theory}, 39(1):78--83, 1993.

\bibitem{Pac00}
F.~Paccaut.
\newblock Statistics of return times for weighted maps of the interval.
\newblock {\em Ann. Inst. H. Poincar\'e}, 36:339--366, 2000.

\bibitem{Pe97}
Ya. Pesin.
\newblock {\em Dimension theory in dynamical systems}.
\newblock University of Chicago Press, 1997.
\newblock Contemporary views and applications.

\bibitem{Pi07}
V.~Pinheiro.
\newblock Expanding measures.
\newblock 2009.
\newblock Preprint http://arxiv.org/abs/0811.2545.

\bibitem{Sau01}
B.~Saussol.
\newblock On fluctuations and exponential statistics of return times.
\newblock {\em Nonlinearity}, 14:179--191, 2001.

\bibitem{Sau06}
B.~Saussol.
\newblock Recurrence rate in rapidly mixing dynamical systems.
\newblock {\em Discrete Contin. Dyn. Syst.}, 15:259--267, 2006.

\bibitem{STV02}
B.~Saussol, S.~Troubetzkoy, and S.~Vaienti.
\newblock Recurrence, dimensions and {L}yapunov exponents.
\newblock {\em Jour. Stat. Phys.}, 106:623--634, 2002.

\bibitem{STV03}
B.~Saussol, S.~Troubetzkoy, and S.~Vaienti.
\newblock Recurrence and {L}yapunov exponents.
\newblock {\em Mosc. Math. J.}, 3:189--203, 2003.

\bibitem{Va08b}
P.~Varandas.
\newblock Correlation decay and recurrence asymptotics for some robust
  nonuniformly hyperbolic maps.
\newblock {\em J. Stat. Phys.}, 133:813--839, 2008.

\bibitem{VV1}
P.~Varandas and M.~Viana.
\newblock Existence, uniqueness and stability of equilibrium states for
  non-uniformly expanding maps.
\newblock Submited.

\bibitem{You82}
L.-S. Young.
\newblock Dimension, entropy and {L}yapunov exponents.
\newblock {\em Ergod. Th. {\&} Dynam. Sys.}, 2:109--124, 1982.

\bibitem{Yu99}
M.~Yuri.
\newblock Thermodynamic formalism for certain nonhyperbolic maps.
\newblock {\em Ergod. Th. {\&} Dynam. Sys.}, 19:1365--1378, 1999.

\end{thebibliography}

\end{document}